\documentclass[english,11pt]{amsart}

\usepackage[utf8]{inputenc}
\usepackage{babel}
\usepackage{amsmath}
\usepackage{amsthm}
\usepackage {amssymb}
\usepackage{hyperref}

\newtheorem{theorem}{Theorem}[section]
 
\newtheorem{lemma}[theorem]{Lemma}
\newtheorem{remark}[theorem]{Remark}
\newtheorem{corollary}[theorem]{Corollary}
\newtheorem{definition}[theorem]{Definition}

\newcommand{\be}{\begin{equation}}
\newcommand{\ee}{\end{equation}}
\newcommand {\R}{\mathbb{R}}

\makeatother
\begin{document}

\title[Symmetry breaking for the Fractional Laplacian] {Symmetry breaking for
an elliptic equation involving the Fractional Laplacian}
\author{Pablo L. De Nápoli}
\address{IMAS (UBA-CONICET) and Departamento de Matem\'atica, Facultad de Ciencias Exactas y Naturales, Universidad de Buenos Aires, Ciudad Universitaria, 1428 Buenos Aires, Argentina}
\email{pdenapo@dm.uba.ar}

\thanks{Supported by ANPCyT under grant PICT 1675/2010, by CONICET under  grant PIP 1420090100230 and by Universidad de
Buenos Aires under grant 20020090100067. The author is a member of CONICET, Argentina.}


\maketitle

\begin{abstract}
We study the symmetry breaking phenomenon for an elliptic equation involving
the fractional Laplacian in a large ball. Our main tool is an 
extension of the Strauss radial lemma involving the fractional
Laplacian, which might be of independent interest; and from which we derive
compact embedding theorems for a Sobolev-type space of 
radial functions with power weights.
\end{abstract}

\section{Introduction}

Recently, elliptic problems involving the fractional Laplacian operator
$(-\Delta)^{s}$ (see section \ref{definitions} for its
definition) have received a great deal of attention. Due to the rotation invariance of the fractional Laplacian it makes sense to ask if the solutions of equations of the form
\be (-\Delta)^{s} u = f(|x|,u)  \label{general-equation} \ee
in a radially symmetric domain (for instance, with homogeneous Dirichlet
conditions) are necessarily radial or not.

\medskip

In this work, we investigate the problem
\be (-\Delta)^{s} u + |x|^a |u|^{q-2} u = |x|^b |u|^{p-2} u \label{our-equation} \ee
with $a,b>0$, and $u$ in the natural energy space for this problem, 
$H^{s}_{q,a}(\R^n)$ (see section \ref{definitions} for 
the precise definitions), and analogous problems in a ball  
$B_R =\{ x \in \R^n : |x| < R\}$ (for $q=2$). 

\medskip

Our aim is to extend to the fractional Laplacian setting, 
the results on existence of radial solutions and symmetry breaking obtained 
by P. Sintzoff in \cite{Sintzoff} 
on this equation \eqref{our-equation}  for the case $s=1$ 
(that of the usual Laplacian). Namely, our main result on symmetry breaking reads as 
follows:

\begin{theorem}
Let $n \geq 2$, $1/2 < s < 1$, $2<p<2^*=\frac{2n}{n-2s}$, $0 < a<n$ and 
$b>\frac{ap}{2}$. If in addition,
\be 
 a(p-2-2ps) + 4bs  <  2s(p-2)(n-1),
\label{compactness-condition}  
\ee
then for every $R>0$ large enough, problem 
\be  \left\{
\begin{gathered}
(-\Delta)^{s} u + |x|^a u = |x|^b u^{p-1} \quad \hbox{in} \; B_R \\
u>0 \;  \hbox{a.e. in}  \; B_R, \quad
u \equiv 0 \;  \hbox{in} \; \R^n-B_R
\end{gathered}
\right. \label{inaball} 
\ee
has a nontrivial radial weak solution and a nonradial one (in
the natural energy space $H^s_{q,a,0,rad}(B_R)$ for this problem, see 
section \ref{definitions}) .
\label{theorem-symmetry-breaking}
\end{theorem}

This result extends theorem 3.1 in \cite{Sintzoff} to the fractional Laplacian
setting. We follow the proof there, which is based in the comparison of the
energy levels for the radial and non-radial ground states (minimizers of the
associated energy functional). However, some
technical difficulties arise from the non-local character of the fractional
Laplacian. For instance, we shall require  a technical lemma from \cite{Raman} in order to
perform the arguments based on cut-off functions. Moreover, we shall need to
prove a version of Strauss inequality adapted to this problem (see the
discussion below).

\medskip
There is a large literature on the subject of radial symmetry of solutions
of elliptic equations, starting from the classical result of B. Gidas, W.M. Ni and L. Nirenberg,  \cite{GNN} on the radial symmetry of positive solutions of \eqref{general-equation} in a ball, for the usual Laplacian ($s=1$). Hence, we cannot attempt to give a complete list of references here. In that work, the authors employed the moving plane method of A. D.  Alexandrov and J. Serrin. We recall that  this kind of result typically  requires  $f$ to be decreasing as a function of $|x|$ (see theorem $1^\prime$ of \cite{GNN}). For the fractional Laplacian, the radial symmetry of positive solutions of some equations like \eqref{general-equation} has been investigated by using the moving plane method  in \cite{FW} and \cite{FQT} (in both cases for $f$ independent of $x$).   

\medskip

On the other hand, when $f$ is not decreasing as a function of $|x|$, a symmetry breaking phenomenon may occur: equations of the form \eqref{general-equation} may admit positive non-radial solutions. A case that has been extensively studied is that of the Henón equation 
$$ -\Delta u = |x|^b u^{p-1} \quad b>0 $$
in a ball (\cite{Smets-Su-Willem}, \cite{Serra}). Hence it is reasonable to 
expect equation \eqref{our-equation} to exhibit a similar behavior.
Other related results on symmetry breaking for nonlinear elliptic equations include 
\cite{CW} (for a singular elliptic equation related to the
Caffarelli-Kohn-Nirenberg inequalities), \cite{FBLR} and \cite{LT} (where
the symmetry breaking phenomenon is investigated for the minimizer of the
trace inequality), and \cite{HC} where the results of \cite{Sintzoff} were 
extended to equations involving the p-Laplacian.

\medskip

There is also an increasing literature on Schr\"odinger equations with the
fractional Laplacian (\cite{Secchi}, \cite{Secchi}, \cite{FQT}, 
\cite{FLS}, \cite{DPV2} among other works). However, the question of symmetry breaking
for the ground states for \eqref{inaball} (or related equations involving
the fractional Laplacian) seems not to have been studied before.

\medskip

As usual, when studying problems in $\R^n$ by variational methods, it is essential to get
some compactness. It is well known that the Sobolev embedding
$$ H^s(\R^n) \subset L^p(\R^n) \; 2 \leq p \leq 2^*_s=\frac{ns}{n-2s} $$
is not compact, due to the translation invariance of the norms of those spaces. 
However, starting by the pioneering work of W. Strauss
\cite{Strauss}, it is known that one can get compactness (for $2<p<2^*$ by restricting the
problem to the subspace of radial functions. More precisely, W. Strauss 
proved in \cite{Strauss} that the following inequality holds for all
\emph{radially symmetric} functions in $H^1(\R^n)$.
$$ 
\| u \| \leq C \; |x|^{-(n-1)/2} \; \| u \|_{H^1(\R^n)} \quad \quad n \geq 2, |x|\geq 1
$$
This inequality means that radially symmetric functions in $H^1(\R^n)$
necessarily have a certain decay rate at infinity, and implies the required
compactness in \cite{Strauss}. Strauss result was later generalized in many directions, 
see \cite{DD6} for a survey of related results and elementary proofs of some of them within the 
framework of potential spaces.

\medskip

In this work, our main tool will be a version of the Strauss inequality for the energy space 
$H^{s}_{q,a}(\R^n)$ of our problem (theorem \ref{Strauss-generalized}) which might be of independent interest. The proof is
completely different from the analogous result for the case of a local operator ($s=1$) in
\cite{Sintzoff}, which was based on a simple integration by parts argument.
Instead, we use some ideas from harmonic analysis, namely: we split the
function into a high and a low frequency part. As a Strauss-type inequality implies the 
continuity of radial functions in that functional space outside the origin, we cannot expect this type 
of result to hold unless $s>1/2$, and that is why this restriction appears in theorem 
\ref{theorem-symmetry-breaking}. 

\medskip

This paper is organized as follows: In section \ref{definitions} we collect the basic 
definitions and notations that we use. In section, \ref{Strauss-section} we state and prove  
the generalization of Strauss inequality for the space $H^{s}_{q,a}(\R^n)$.
In section \ref{Holder-continuity-section}, we show that the same type of estimates
allows us to prove the H\"older continuity of the functions in that space,
outside the origin. This will be useful to us in order to be able to use the
Arzela-Ascoli theorem for the compactness arguments. In section
\ref{embeddings-section}, we derive from our Strauss-type inequality,  
a compact embedding result for radial
functions in $H^s_{q,a}(\R^n)$ into $L^p$ spaces with power weights. As a corollary, we obtain some existence
results for radial solutions of \eqref{our-equation} in the whole space
$\R^n$. Finally, in section \eqref{sb-section} we prove theorem
\ref{theorem-symmetry-breaking}.

\section{Basic definitions and notations}

\label{definitions}

In this section, we define the functional spaces that we are going to work with, and recall some basic facts about the fractional Laplacian.

\medskip

The fractional Laplacian $(-\Delta)^{s}$ can be defined for  functions $u \in \mathcal{S}(\R^n)$ (the Schwartz class) 
and $0<s<1$ by means of the Fourier transform as
\be \widehat{(-\Delta)^{s} u(\omega)} = 
|\omega|^{2s} \; \widehat{u}(\omega). \label{definition-Fourier} \ee

Alternatively, we may define  the fractional Laplacian by a hypersingular integral
$$ (-\Delta)^{s} u(x)= C(n,s) \; \hbox{P.V.} \; \int_{\R^n} \frac{u(x)-u(y)}{|x-y|^{n+2s}} \; dy $$
where
$$ C(n,s)= \left( \int_{\R^n} \frac{1-\cos(\zeta_1)}{|\zeta|^{n+2s}} \; d\zeta \right)^{-1} $$
is a normalization constant. It is also possible to define the fractional Laplacian by means of an extension problem (see \cite{CS}).

\medskip

For $0<s<1$, we consider the Gagliardo seminorm, 
\be [u]_{H^s} =  \| (-\Delta)^{s/2} u \|_{L^2} = 
\left(  \frac{C(n,s)}{2} \int_{\R^n} \int_{\R^n}  
\frac{|u(x)-u(y)|^2}{|x-y|^{n+2s}} \; dx \; dy \right)^{1/2} 
\label{Gagliardo-seminorm} \ee	
which, for $u \in \mathcal{S}(\R^n) $, can be also expressed in terms of the Fourier transform
$$ [u]_{H^s} = \left( \int_{\R^n} |\omega|^{2s} |\widehat{u}(\omega)|^2 \; 
d\omega \right)^{1/2}. $$
We refer to \cite{DPV} for more details.

\medskip

We denote by $L^q_a(\R^n)$ the weighted Lebesgue space with the norm as
$$ \| u \|_{L^q_a}= \left( \int_{\R^n} |u(x)|^q \; |x|^a \; dx 
\right)^{1/q}. $$

Then, we define the space $H^{s}_{q,a}(\R^n)$, which will be the natural energy space for problem \eqref{our-equation}, as the completion of $C_0^\infty(\R^n)$ with respect to the norm
\be \| u \|_{H^{s}_{q,a}} = \left( [u]_{H_s}^2 + 
\| u \|_{L^q_a}^2 \right)^{1/2}. \; \label{norm-definition} \ee

Moreover, we denote by $H^s_{q,a,rad}(\R^n)$  the subspace of radial 
functions in  $H^s_{q,a}(\R^n)$. 

In a similar way, given a domain $\Omega \subset \R^n$ we denote by 
$H^s_{q,a,0}(\Omega)$ the closure of $C_0^\infty(\Omega)$ in 
$H^s_{q,a}(\R^n)$. And by $H^s_{q,a,0,rad}(\Omega)$ its subspace consisting of 
radial functions, when $\Omega$ is radially symmetric.

When $q=2$ and $a=0$, we have that $H^{s}_{q,a}(\R^n)= H^s(\R^n)$
(the usual fractional Sobolev space, sometimes also denoted by
$W^{s,2}(\R^n)$, see \cite{DPV}).

\medskip

We recall that we have the following \emph{fractional Sobolev inequality} (see for 
instance \cite{DPV}, theorem 6.5):

\begin{theorem} 
Let $0 < s < \min(\frac{n}{2},1)$ and define the Sobolev critical exponent 
\be 2^*_s= 2^*(n,s)=\frac{2n}{n-2s} \label{critical-exponent}. \ee
Then, there exists a constant $C=C(n,s)$ such that 
$$ \| u \|_{L^{2^*_s}} \leq C \; [u]_{H^s} \quad \hbox{for all } \; u \in
H^s(\R^n). $$
\end{theorem}

\begin{corollary}
Let $a \geq 0$ and $0 < s < \min(\frac{n}{2},1)$. Then
$$ H^{s}_{2,a}(\R^n) \subset  H^{s}(\R^n). $$
\label{coro-inL2}
\end{corollary}

\begin{proof}
It suffices to show that
\be \| u \|_{L^2} \leq C [u]_{H^s} \quad \; \forall \; u \in
C_0^\infty(\R^n). \label{inL2} \ee
However,
$$ \int_{|x|> 1} u^2 \: dx \leq \int_{|x|> 1} |x|^a u^2 \: dx $$
and
$$ \int_{|x| \leq  1} u^2 \: dx \leq C \left( \int_{|x|\leq 1} u^{2^*} \: dx \right)^{2/2^*} 
\leq C [u]_{H^s}. $$
Inequality \eqref{inL2} follows.
\end{proof}

\begin{remark}
It is desirable to have a more concrete characterization of the functional space 
$H^s_{q,a,0}(\Omega)$. For this propose, we consider the fractional homogeneous 
Solev space
$$ \dot{H}^s(\R^n)= \{ u \in L^{2^*_s}(\R^n) : [u]_{H^s} < \infty \}.  $$
Then it is well known that $C_0^\infty(\R^n)$ is dense in $\dot{H}^s(\R^n)$ 
(This for instance the case $a=0$ of theorem 1.1 in \cite{DV}, which is established 
by the standard procedure of truncation and regularization by convolution with a 
standard molifier). Then,
$$ H^{s}_{q,a}(\R^n) = \dot{H}^s(\R^n) \cap L^q_a(\R^n) \; \hbox{if} \; 
-n < a < n(q-1)$$
(We need the restrictions on $a$ so that $|x|^a \in A_q$, the Muckenhoupt class of
 weights, and the regularization by convolution with a standard molifier works, see 
 for example lemma 1.5 in \cite{K}). Likewise, if $\Omega$ is a ball
$$ H^{s}_{q,a}(\Omega) = \{ u \in  \dot{H}^s(\R^n) \cap L^q_a(\R^n) : u = 0 \; \hbox{a.e} \; \hbox{in} 
\; \Omega^c \; \} \quad \hbox{for} \; 
-n < a < n(q-1). $$
For some other similar density results for fractional Sobolev spaces 
in domains see \cite{FSV}.
\label{remark-characterizations}
\end{remark}

\medskip

Since we are going to use variational methods, it is natural to interpret
our solutions in a weak sense:	

\begin{definition}
We say that $u \in H^s_{q,a,0}(\Omega)$ is a weak solution of
$$  \left\{
\begin{gathered}
(-\Delta)^{s} u + |x|^a |u|^{q-2}u  = |x|^b u^{p-2} u \quad \hbox{in} \;\Omega \\ 
u \equiv 0 \;  \hbox{in} \; \R^n-\Omega
\end{gathered}
\right.  
$$
provided that 
\begin{multline*}
C(n,s) \int_{\R^n} \int_{\R^n} \frac{(u(x)-u(y))}{|x-y|^{n+2s}} \cdot 
(\varphi(x)-\varphi(y)) \; dx \; dy + \int_{\Omega}
 |x|^a |u|^{q-2} u \; \varphi \; dx  \\=  \int_{\Omega} |x|^b |u|^{p-2} u \;
\varphi \; dx 
\end{multline*}
holds for any test function $\varphi \in C_0^\infty(\Omega)$.  
\label{definition-weak-solution}
\end{definition}

\medskip

\begin{remark}
It is important to notice that different notions of fractional Laplacian in
domains $\Omega \subset \R^n$, with different interpretations of the
Dirichlet boundary conditions, have been defined in the literature, which
should not be confused.

\begin{enumerate}
\item[i)] We may consider equations with the standard fractional Laplacian on $\R^n$ and require
that the solutions vanish outside $\Omega$. The associated \emph{Dirilect
form} is

$$ \mathcal{E}_{\R^n}(u,v) = \frac{C(n,s)}{2} \int_{\R^n} \int_{\R^n}
\frac{(u(x)-u(y))\cdot (v(x)-v(y))}{|x-y|^{n+2s}} \; dx \; dy. $$ 

In probabilistic terms, the associated stochastic process is the standard symmetric $\alpha$-stable
Lévy process (with $\alpha=2s$), killed upon living $\Omega$. This is the operator that
we consider in definition \ref{definition-weak-solution} and theorem \ref{theorem-symmetry-breaking}, and it has been also
used for instance in \cite{LL} (in even a more general form, called there
the fractional $p$-Laplacian).

\item[ii)] A second option is to consider the so-called regional fractional Laplacian,
with corresponds to the Dirichlet form

$$ \mathcal{E}_\Omega (u,v) = \frac{C(n,s)}{2} \int_{\Omega} \int_{\Omega}
\frac{(u(x)-u(y))\cdot (v(x)-v(y))}{|x-y|^{n+2s}} \; dx \; dy. $$ 

The associated stochastic process is the censored stable process, for which
jumps outside $\Omega$ are completely forbidden \cite{BBC}.

\item[iii)] Another approach is to consider the fractional powers of the Laplacian in
$\Omega$ with Dirichlet conditions, defined from the spectral decomposition. 
This coincides with the operator obtained from the Caffarelli-Silvestre
extension on a cylinder based in $\Omega$ (see \cite{BCDP}). The
associated stochastic process is the subordinate killed Brownian motion
studied in \cite{SV}.
\end{enumerate}

We don't know if the analogue of theorem \ref{theorem-symmetry-breaking}
holds for the fractional Laplacian in $\Omega$ in the sense of ii) or iii).
\end{remark}

\medskip

We shall need the following strong minimum principle for weak
supersolutions, which can be found in \cite{BF}. We start by a definition:

\begin{definition}
We say that $u \in H^{s,2}_0(\Omega)$ is a weak supersolution of
\be (-\Delta)^s u = 0 \; \hbox{in} \; \Omega \quad u \equiv 0 \; \hbox{in} \; \Omega - \R^n 
\label{equation-homog} \ee 
if  
$$ \int_{\R^n} \int_{\R^n} \frac{(u(x)-u(y))}{|x-y|^{n+2s}}\;
(\varphi(x)-\varphi(y)) \; dx \; dy \geq 0 $$
for all $\varphi \in C_0^\infty(\Omega)$, $\varphi \geq 0$.
\end{definition}

Then we may state the result:

\begin{theorem}[\cite{BF}, theorem A.1, case $p=2$]
Let $\Omega \subset \R^n$ be an open bounded set, which is connected. Let $s
\in (0,1)$ and $u \in H^{s}_0(\Omega)$ be a weak supersolution of \eqref{equation-homog}
in the sense of the previous definition  such that 
$u \geq 0$ in $\Omega$. Let us suppose that
$$ u \not \equiv 0  \; \hbox{in} \; \Omega. $$
Then $u>0$ almost everywhere in $\Omega$
\label{minimum-principle}.
\end{theorem}

We shall need also the following elementary calculus lemma, whose proof is 
straightforward. 
\begin{lemma}
Consider the function $f:(0,\infty) \to \R$ 
$$ f(\lambda)=C_1 \; \lambda^e_1 + C_2 \; \lambda^{-e_2} $$
where $C_1,C_2,e_1,e_2>0$. Then $f$ archives its minimum at the point
$$ \lambda_0 = \left( \frac{ C_2 \; e_2 }{ C_1 \; e_1}\right)^{1/(e_1+e_2)}  $$
and
$$ f(\lambda_0)= C_1^{e_2/(e_1+e_2)} \; C_2^{e_1/(e_1+e_2)} \; k(e_1,e_2) $$
where $k(e_1,e_2)$ depends only on the exponents $k_1$ and $k_2$.
\label{minimization-lemma}
\end{lemma}

\section{A generalization of Strauss Radial Lemma}

\label{Strauss-section}

In this section, we shall prove a version of Strauss 
radial lemma for the space $H^s_{q,a,rad}(\R^n)$, generalizing lemma 2.2 in \cite{Sintzoff}. 

\begin{theorem}
Assume that $s>\frac{1}{2}$, and that
$-(n-1) \leq a < n(q-1)$. Define the exponents
$$
\theta= \theta(s,q)= \frac{2}{2sq+2-q} \quad (0 < \theta < 1)
$$
and
$$
\sigma = \theta \; \frac{n-1}{2}  +  (1-\theta) 
\; \frac{n-1+a}{q} \; =  \; \frac{2 \, a s + 2 \, n s - a - 2 \, s}{2 \, q s - q +
2}.
$$


For any radial function $u \in H^{s}_{q,a,rad}(\R^n)$, we have that
\be |u(x)| \leq C(n,s,q,a) \;  |x|^{-\sigma} \; 
[u]_{H^s}^{\theta} \; 
 \| u \|_{L^q_a}^{1-\theta}.  \label{Strauss-ineq-generalized} \ee 
\label{Strauss-generalized}

\medskip

As a consequence, any function $u \in H^s_{q,a,rad}(\R^n)$ is equal a.e. to 
a continuous function in $\R^{n}-\{0\}$ and we have that

\be |u(x)| \leq C(n,s,q,a) \;  |x|^{-\sigma} \;  \| u \|_{H^s_{q,a}}.
\label{Strauss-ineq2} \ee

\end{theorem}

\medskip

For the proof we need the following lemmas.

\begin{lemma}[\cite{SW},theorem 3.3 of chapter IV]
Let $u \in L^1_{rad}(\R^n)$ be a radial function, $u(x)=u_{0}(|x|)$.
Then its Fourier transform $\widehat{u}$ is also radial, and it is given by 
$$ \widehat{u}(\omega) = (2\pi)^{n/2} |\omega|^{-\nu} \int_0^\infty u_{0}(r)
\; J_{\nu}(r|\omega|) \; r^{n/2}\; dr $$
where $\nu=\frac{n}{2} -1$ and $J_\nu$ denotes the Bessel function of
order $\nu$.
\label{radial-Fourier}
\end{lemma}

\begin{remark}
If $u$ is a radial function, in particular is even ($u(-x)=u(x)$). It
follows than the inverse Fourier transform of $u$, coincides with the Fourier
transform. In other words, for radial functions $u \in L^1_{rad}(R^n)$, we can write Fourier
inversion formula as
\be u(x) = (2\pi)^{n/2} |x|^{-\nu} \int_0^\infty (\widehat{u})_{0}(r)
\; J_{\nu}(r|x|) \; r^{n/2}\; dr \; \label{inversion-radial} \ee 
\end{remark}

\medskip

\begin{lemma}[Assymptotics of Bessel functions, \cite{SW}, lemma 3.1 of
chapter IV]
If $\lambda>-1/2$, then
\be |J_\lambda(r)| \leq C \; r^{-1/2} \label{Bessel-bound}. \ee
\label{lemma-Bessel}
\end{lemma}

\begin{lemma}
Let $\gamma>n-\frac{1}{p}$ and consider the operator
$$ S_\gamma f(x)= \int_{\R^n} \frac{f(y)}{(1+|x-y|^2)^{\gamma/2}} \; dy. $$
Then if $-\frac{n-1}{p}<\alpha<\frac{n}{p^\prime}$, there exists $C>0$ such that
for any radial function $f$,
$$ |S_\gamma f(x)| \leq C |x|^{-(n-1)/p-\alpha} \; \| |x|^\alpha f \|_{L^p}. $$
\label{S-gamma}
\end{lemma}

\begin{proof}
This is a special case $\beta=-\frac{n-1}{p}-\alpha$, $q=\infty$ and 
$\tilde{p}=\tilde{q}=p$ of lemma 2.3 in \cite{DL}. 
\end{proof}

Now we proceed to the proof of theorem \ref{Strauss-generalized}.

\begin{proof}
Let $u \in C_0^\infty(\R^n)$. We choose a radial function 
$\psi \in C_0^\infty(\R^n)$ such that $\psi \equiv 1 $ in the ball
$B(0,1)$. Moreover, we consider $\phi=\widehat{\psi}$ which will be a radial function
in the Schwartz class $\mathcal{S}(\R^n)$, and we use it to split $u$ into low 
and high frequency parts, as
$$ u(x)=h(x)+l(x) $$
where $l(x)=(u*\phi_t)(x)$ with $\phi_t(x)=t^{-n} \phi(x/t)$. We call $l$
the low  frequency part of $u$, since its Fourier transform
$$
\widehat{l }(\omega)= \widehat{u}(\omega) \psi(t\omega) 
$$
is supported in the ball $B(0,1/t)$. In a similar way,
\be \widehat{h}(\omega)= \widehat{u}(\omega) (1-\psi(t\omega) )
\label{h-transform} \ee
is supported in the region $|\omega|\geq 1/t$.

\medskip

For the high frequency part, $h$ we use an estimate on the frequency side. Indeed, writing the 
Fourier inversion formula \eqref{inversion-radial} for $h$, 
and using the bound \eqref{Bessel-bound} for the Bessel
functions, we get that: 
\be |h(x)| \leq C_n \; |x|^{-(n-1)/2} \int_0^\infty
|(\widehat{h})_{0}(r)| \; r^{(n-1)/2} \; dr.  \ee

Hence, using the Cauchy-Schwarz inequality and \eqref{h-transform}, we see that
\begin{align*}
|h(x)|&  \leq C_n \; |x|^{-(n-1)/2} \left( \int_0^\infty
|(\widehat{u})_{0}(r)|^2 \; r^{2s} r^{n-1} \; dr  \right)^{1/2}
\left( \int_0^\infty \left| 1-\psi_{0}(tr) \right|^2 r^{-2s} \; dr
\right)^{1/2}\\
&  \leq C_n \; |x|^{-(n-1)/2}  \left( \int_{\R^n} |\widehat{u}(\omega)|^2 
|\omega|^{2s} \; d\omega \right)^{1/2}
\left( \int_{1/t}^\infty \left|1- \psi_{0}(tr) \right|^2 r^{-2s} \; dr
\right)^{1/2}\\
&\leq C_n \; |x|^{-(n-1)/2} \; t^{s-1/2}  \; [u]_{H^s} \;
\left( \int_{1}^\infty \left| 1-\psi_{0}(z) \right|^2 z^{-2s} \; dz
\right)^{1/2}\\
&\leq C \; |x|^{-(n-1)/2} \; t^{s-1/2} \; [u]_{H^s}  
\end{align*}
as the last integral is finite, since $s>1/2$ by hypothesis.

\medskip

Now we consider the low frequency part, and we use an estimate on the space
side. We fix $\gamma > (n -1+a)/q$. Then, since $\phi \in 
\mathcal{S}(\R^n)$, there exist $C>0$ such that 
$$ |\phi(x)| \leq C (1+|x|^2)^{-\gamma/2}. $$
Then we estimate $l$ as follows,
\begin{align*}
|l(x)| & \leq |u(x)| * |\phi_t(x)| \leq C  \int_{\R^n}
\frac{1}{t^n} \frac{|u(y)|}{\left(1+\left|\frac{x-y}{t}\right|\right)^{\gamma/2}} \; dy \\
& \leq C \int_{\R^n} \frac{|u(tz)|}{\left(1+\left| \frac{x}{t}-z \right|^2 \; dz
\right)^{\gamma/2}} \;  \; dz \quad (y=tz) \\
& \leq  C  S_\gamma(u_t)\left(\frac{x}{t}\right) \quad \hbox{where} \; u_t(z)=u(tz) \\
& \leq C \; \left|\frac{x}{t}\right|^{-(n-1+a)/q} \; \left( \int_{\R^n} 
|u(tz)|^q \; |z|^{q}
\; dz \right)^{1/q} \; \quad \hbox{by lemma} \; \ref{S-gamma} \;
(\hbox{with} \; p=q, \alpha=a/q ) \\
& \leq C \; \; t^{(n-1+a)/q} \; \left|x \right|^{-(n-1+a)/q}  \left( \int_{\R^n} 
|u(y)|^q \; \left|\frac{y}{t}\right|^{q} \; \frac{dy}{t^n}
\right)^{1/q}\\
& \leq C \; t^{-1/q} \;  \left|x \right|^{-(n-1+a)/q} \; \| u \|_{L^q_a}.
\end{align*}
Therefore, collecting our estimates, we have that
\begin{align*}
|u(x)| &\leq |h(x)|+|l(x)| \\
& \leq
 C \;  \left[\; |x|^{-(n-1)/2} \; t^{s-1/2}  \; [u]_{H^s} +  \;  
|x|^{-(n-1+a)/q} \;  t^{-1/q} \; \| u \|_{L^q_a} \right].
\end{align*}

We choose $t$ in order to minimize the right hand side of this expression, using 
lemma \ref{minimization-lemma}, with 
$$ e_1 = s-1/2, $$
$$ e_2 = 1/q, $$
$$ C_1 =  |x|^{-(n-1)/2} \; [u]_{H^s}, $$
$$ C_2 = |x|^{-(n-1+a)/q} \; \| u \|_{L^q_a}. $$

\noindent We obtain  \eqref{Strauss-ineq-generalized}, and \eqref{Strauss-ineq2} follows
immediately.

\medskip

Now let $u \in H^s_{q,a,rad}(\R^n)$ and consider a sequence of radial functions $(u_k)$ in
$C_0^\infty(\R^n)$ that $u_k \to u$ in $H^s_{q,a,rad}(\R^n)$. Then, 
$$ | u_k(x) - u_j(x) | \leq C \; |x|^{-\sigma} \;  \| u_k - u_j \|_{H^s_{q,a}} $$
It follows that $(u_n)$ converges uniformly on compact subsets of
$\R^{n}-\{0\}$ 
to a continuous function, which can be taken as a continuous representative
of $u$, and which satisfies inequality \eqref{Strauss-ineq-generalized}. 
\end{proof}

\begin{remark}
When $q=2$ and $a=0$, we get the Strauss inequality for the usual fractional
Sobolev space $H^s(\R^n)$ (for $s>\frac{1}{2}$ )
\be |u(x)| \leq C(n,s) \;  |x|^{-(n-1)/2} \; 
[u]_{H^s}^{\frac{1}{2s}} \;  \| u
\|_{L^2}^{1-\frac{1}{2s}}  \label{Strauss-fraccionario-1} \ee
and as a consequence
\be |u(x)| \leq C(n,s) \;  |x|^{-(n-1)/2} \; \| u \|_{H^s(\R^n)}. 
\label{Strauss-fraccionario-2} \ee

\;

Indeed in this particular case, the above proof can be simplified. We can take 
$\psi(x)$ as the characteristic function of the unit ball (which is not 
$C^\infty$, but we don't need it for this alternative proof), and estimate 
the low frequency part $l$ on the Fourier side, by using lemma \ref{radial-Fourier}, lemma
\ref{lemma-Bessel} and  Plancherel theorem (as we did before for $h$).

\begin{align*}
|l(x)| & \leq C_n \; |x|^{-(n-1)/2} \; \int_0^\infty
|(\widehat{l})_{0}(r)| \; r^{(n-1)/2} \; dr \\
&  \leq C_n \; |x|^{-(n-1)/2} \;  \int_0^{1/t}
|(\widehat{u})_{0}(r)| \; r^{(n-1)/2} \; dr \\
&  \leq C_n \; |x|^{-(n-1)/2} \; t^{-1/2} \;  \left( \int_0^{1/t}
|(\widehat{u})_{0}(r)|^2 \; r^{n-1} \; dr \right)^{1/2} \\  
&  \leq C_n \; |x|^{-(n-1)/2}  \;  t^{-1/2} \| \widehat{u} \|_{L^2} \\ 
&  \leq C_n \; |x|^{-(n-1)/2}  \; t^{-1/2} \| u \|_{L^2} 
\end{align*}
 
and we have arrived to the same estimate for $l$ as before, but without
using lemma \ref{S-gamma}. This is essentially the argument that 
we have used in \cite{DD6} to derive \eqref{Strauss-fraccionario-1}.  
A technique similar to the ours, has been used by Y. Cho and T. Ozawa \cite{CO} 
to derive various related Sobolev inequalities with symmetries. 
However, this simpler approach does not work in the general case of theorem \ref{Strauss-generalized}
(as Plancherel theorem is not available for the weighted $L^q$-norm).

\medskip 

The inequality \eqref{Strauss-fraccionario-2} is a particular case of the
results of  W. Sickel and L. Skrzypczak \cite{SiSk}, who proved a version of 
Strauss lemma for Besov and Tribel-Lizorkin spaces, since $H^s(\R^n)$
coincides with the Tribel-Lizorkin space $F^s_{p,2}(\R^n)$. However, the
above proof is much simpler than the proof in \cite{SiSk}, which is based on
an atomic decomposition.

\end{remark}

\begin{remark}
When $s=1$, $H^{s}_{a,q}(\R^n)$ coincides with the space $X$ considered by 
Sintzoff \cite{Sintzoff}, and as said before,  our result (partially) extends lemma 2.2 in
\cite{Sintzoff} to the
fractional case $s>1/2$. 

\medskip

However, Sintzoff's proof (which is based like the
original proof in \cite{Strauss} in an integration by parts argument), does
not have the restriction $a<(n-1)q$ (that comes from the use of lemma 
\ref{S-gamma}). We conjecture that this restriction is not necessary for
theorem \ref{Strauss-generalized} to hold, but we have not been able to remove it.
 
\end{remark}

\section{H\"older continuity estimates}

\label{Holder-continuity-section}

In this section, we show that the same kind of estimates in the proof of theorem 
\ref{Strauss-generalized}, can be used to obtain local H\"older continuity 
of the functions in $H^s_{q,a,rad}(\R^n)$. This fact will be useful later
for proving compactness of the embedding of this space into weighted $L^p$ spaces. 

\begin{theorem}
Let $\Omega_\varepsilon = \{ x \in \R^n: |x|\geq \varepsilon\}$. Then the continuous representative of
a function $u\in H^s_{q,a,rad}(\R^n)$ is H\"older  continuous in
$\Omega_\varepsilon$, and moreover there exists a constant $C_\varepsilon>0$
such that
\be |u(x_1)-u(x_2)| \leq C_\varepsilon \;  |x_1 - x_2|^\alpha \; 
\| u \|_{H^s_{q,a}}(\R^n) \label{Holder-continuity-equation} \ee
with
\be  \alpha=\frac{s-1/2}{s+1/q-1/2} \label{def-alpha}. \ee
\label{Holder-continuity} 
\end{theorem}

\begin{proof}
We use the same decomposition $u=h+l$ as in the proof of theorem
\ref{Strauss-generalized}. For the low frequency part $l$, we know hat 
$ \nabla l(x) = u*\nabla(\phi_t)(x)$. Hence,
$$ |\nabla l(x)| \leq C \frac{1}{t^{n+1}} \; 
\int_{\R^n} |u(y)| \;  \left|\nabla \phi\left(\frac{x-y}{t}\right)\right| 
\; dy. $$
Now all the partial derivatives $\frac{\partial \phi}{\partial x_i}$ are also in 
$\mathcal{S}(\R^n)$. Hence estimating this integral using lemma
\ref{S-gamma}  as before, we get
$$ |\nabla l(x)| \leq C \;  t^{-1/q-1} \;  \left|x \right|^{-(n-1+a)/q} \; 
\| u  \|_{L^q_a}. $$
Recalling that $l$ is radial, and using the mean value theorem, we get
$$ |l_0(\rho_1)-l_0(\rho_2)| \leq C_\varepsilon \;  t^{-1/q-1} \; |\rho_1 -
\rho_2| \;  \| u  \|_{L^q_a}
\quad \hbox{for all}\; \rho_1,\rho_2 \geq \varepsilon. $$

\medskip

Now, for the high frequency part, using the estimates in the proof of
theorem \ref{Strauss-generalized}, we have that
$$
|h_0(\rho_1)-h_0(\rho_2)| \leq |h(\rho_1)| + |h(\rho_2)| \\
\leq  C_\varepsilon \; t^{s-1/2} \; [u]_{H^s} 
$$
Collecting our estimates, we find that
$$ |u_0(\rho_1)-u_0(\rho_2)| \leq C_\varepsilon \left[
 t^{s-1/2} \; [u]_{H^s}  + t^{-1/q-1} \; |\rho_1 -
\rho_2| \;  \| u  \|_{L^q_a} \right] $$
Choosing $t$ to minimize the right hand side according to lemma
\ref{minimization-lemma} as before, gives
$$ |u_0(\rho_1)-u_0(\rho_2)| \leq  C_\varepsilon |\rho_1-\rho_2|^\alpha \;
[u]_{H^s}^{\theta} \| u |x|^{a/q} \|_{L^q}^{1-\theta} $$
with $\alpha$ given by \eqref{def-alpha}, and the same $\theta$ as before. 
We conclude that
$$ |u_0(\rho_1)-u_0(\rho_2)| \leq \;  C_\varepsilon \; H^{s}_{q,a}(\Omega) \subset |\rho_1-\rho_2|^\alpha 
\| u \|_{H^s_{q,a}} $$
and as a consequence, we obtain \eqref{Holder-continuity-equation}.  
\end{proof}

\section{Emebedding theorems}

\label{embeddings-section}

In this section, we prove some embedding theorems for spaces of radial functions. We remark that the power weight $|x|^b$ produces the presence of a shifted Sobolev critical exponent $2^*_b$ in all of them.

\medskip

The following lemma is a special case of the result of \cite{DDD}, that
we state here in our present notation for convenience of the reader, and
generalizes a result due to Rother \cite{Rother} (that corresponds to the case $s=1$). 

\begin{lemma} Let $0<s<\frac{n}{2}$, $c>-2s$, $c(1-2s) \leq 2s(n-1)$, 
$2 \leq q= 2^*_c :=\frac{2(n+c)}{n-2s}$. Then we have that
$$ \left( \int_{\R^n} |x|^c |u|^q \right)^{1/q} \; \leq C \; [u]_{H^s}  $$ 
\label{generalized-Rother}
for any radial function $u \in C_0^\infty(\R^n)$.
\end{lemma}


\begin{proof}
Take $\gamma=n-s$, $p=2$, $\alpha=0$, $\beta=-c/q$ in the result of
\cite{DDD} (theorem 1.2), and recall that $T_{n-s} u= C (-\Delta)^{-s/2} u$, where 
$T_\gamma$ denotes the fractional integral
$$T_\gamma u(x)= \int_{\R^n} \frac{u(y)}{|x-y|^\gamma} \; dy. $$
\end{proof}

We start by proving a Gagliardo-Nirenberg type inequality.

\begin{lemma}
Assume that $\max(q,2) <p <p<2^*_b=\frac{2(n+b)}{n-2s}$,  $a<n(q-1)$ and 
\be
a \, (p-2-2ps) + b \, (2qs - q +2 ) < 
 2s \, (p-q) \, (n-1)  \label{cond-ab}
\ee
Then, there exist $\eta=\eta(n,s,p,q,a,b)$
with $0<\eta<1$ such that for any radial function $u\in
C_0^\infty(\R^n)$
\be \left( \int_{\R^n} |x|^b |u|^p \; dx \right)^{1/p} 
\leq C \; [u]_{H^s}^{\eta} \; 
\left( \int_{\R^n} |x|^a |u|^q \; dx \right)^{(1-\eta)/q}. \label{GN-eq} \ee
\label{GN}
\end{lemma}

\begin{proof}
We devide the integral in $|x|\leq \lambda$, and $|x|\geq \lambda$. We 
define $c$ by $p=\frac{2(n+c)}{n-2s}$. The hypothesis $p<2^*_b$ gives
that $c< b$. Moreover $c>-2s$ since $p>2$. Then lemma \ref{generalized-Rother}, gives that

$$ \int_{|x|\leq \lambda} |x|^b |u|^p \; dx \leq C \lambda^{b-c} 
\; [u]_{H^s}^{p}.  $$

On the other hand, using theorem \ref{Strauss-generalized}, we have that
\begin{align*}
\int_{|x|>\lambda} 
|x|^b \; |u|^p \; dx  & = \int_{|x|>\lambda}
|x|^a \; |u|^q \; |x|^{b-a} \; |u|^{p-q} dx  \\
& \leq C \int_{|x|>\lambda} |x|^a \; |u|^q \; |x|^{b-a-\sigma(p-q)} \;
[u]_{H^s}^{\theta(p-q)} 
\| u  \|_{L^q_a}^{(1-\theta)(p-q)} \; dx \\
& \leq C \lambda^{b-a-\sigma(p-q)} \;
 [u]_{H^s}^{\theta(p-q)} \;
 \| u  \|_{L^q_a}^{(1-\theta)(p-q)+q} 
\end{align*}
with
$$ \sigma =  \; \frac{2 \, a s + 2 \, n s - a - 2 \, s}{2 \, q s - q + 2},
\quad \theta= \frac{2}{2sq+2-q} \quad (0 < \theta < 1)
 $$
provided that $b-a<\sigma(p-q)$ which is equivalent to \eqref{cond-ab}.

We collect our estimates 
$$ \int_{\R^n} 
|x|^b \; |u|^p \; dx \leq
C \left[ \lambda^{b-c} 
\; [u]_{H^s}^{p} + \lambda^{b-a-\sigma(p-q)} \;
 [u]_{H^s}^{\theta(p-q)} \;
 \| u  \|_{L^q_a}^{(1-\theta)(p-q)+q} \right] $$
and optimize for $\lambda$ using lemma \ref{minimization-lemma}, with
$$ e_1 = b-c $$
$$ e_2 = \sigma(p-q)-(b-a) $$
$$ C_1 = [u]_{H^s}^{p} $$
$$ C_2=  [u]_{H^s}^{\theta(p-q)} \;  \| u  \|_{L^q_a}^{(1-\theta)(p-q)+q}. $$
Hence we get \eqref{GN-eq} with
$$ \eta= \frac{e_2}{e_1+e_2} + \theta(1-q/p) \left( 
\frac{e_1}{e_1+e_2}\right). $$
\end{proof}

\begin{theorem}
If $1<q<p$, $2<p<2^*_b=\frac{2(n+b)}{n-2s}$, $a<n(q-1)$ and 
\be
a \, (p-2-2ps) + b \, (2qs - q +2 ) < 
 2s \, (p-q) \, (n-1)  \label{cond-ab2}
\ee
then we have a continuous and compact embedding
$$H^s_{q,a,rad}(\R^n) \subset L^p(\R^n,|x|^b \; dx). $$
\label{compact-embedding} 
\end{theorem}

\begin{proof}
The continuity of the embedding follows from lemma \ref{GN}. To prove the
compactness, we shall follow the argument employed by Sintzoff in 
\cite{Sintzoff}, with some modifications.

\medskip 

Since $H^s_{q,a}(\R^n)$ is reflexive, it is enough to see that for every given 
sequence $(u_n)$ that converges weakly to $0$ in $H^s_{q,a,rad}(\R^n)$, we have that 
$\| u_n \|_{L^p_b} \to 0$.

Since $(u_n)$ is weakly convergent, $(u_n)$ is bounded in $H^s_{q,a}(\R^n)$, 
$$ \| u_n \|_{H^s_{q,a}} \leq M. $$

Given $\varepsilon>0$, we divide the domain $\R^n$ in three parts:
\be \| u_n \|_{L^p_b}^p = \int_{|x|\leq \lambda} |x|^b |u_n|^p \; dx +
 \int_{|x|\geq \frac{1}{\lambda}}  |x|^b |u_n|^p \; dx  +
\int_{ \lambda \leq |x| \leq \frac{1}{\lambda}}  |x|^b |u_n|^p \; dx
\label{decomposition} \ee
where $\lambda=\lambda(\varepsilon)$ will be chosen later.

\medskip

\textbf{First term:} As before, we define $c$ by $p=\frac{2(n+c)}{n-2s}$, where $c<b$. Then by lemma 
\ref{generalized-Rother}, we deduce that
$$ \int_{|x|\leq \lambda} |x|^b \; |u_n|^p \; dx \leq C \lambda^{b-c}
[u_n]_{H^s}^p \leq C \lambda^{b-c} \; M^p <
\frac{\varepsilon}{3} \; \hbox{for} \; \lambda \geq
\lambda_0(\varepsilon).$$

\medskip

\textbf{Second term:}
Again, we use theorem \ref{Strauss-generalized} in this part, recalling that
$b-a<\sigma(p-q)$ is equivalent to \eqref{cond-ab2}:
\begin{align*}
\int_{|x|\geq \frac {1}{\lambda}}  |x|^b |u_n|^p \; dx 
&= \int_{|x|>\lambda}
|x|^a \; |u_n|^q \; |x|^{b-a} \; |u_n|^{p-q} dx  \\
&\leq C \int_{|x|>\lambda} |x|^a \; |u|^q \; |x|^{b-a-\sigma(p-q)} 
[u_n]_{H^s}^{\theta(p-q)} 
\| u_n  \|_{L^q_a}^{(1-\theta)(p-q)} \; dx \\
&\leq C \lambda^{b-a-\sigma(q-p)} \; M^{p-q} \\ 
& < \frac{\varepsilon}{3} \; \hbox{for} \; \lambda \geq
\lambda_1(\varepsilon).
\end{align*}

\textbf{Third term:} Finally, we fix
$\lambda=\max(\lambda_0(\varepsilon),\lambda_1(\varepsilon))$.

\medskip

Consider then, any subsequence $(u_{n_k})$ of $(u_n)$.
From theorems \ref{Strauss-generalized} and \ref{Holder-continuity},
$(u_n)$ is equibounded and equicontinuous in $A_\lambda=\{ x \in \R^n : \lambda \leq
|x| \leq \frac{1}{\lambda} \}$. From the Arzela-Ascoli theorem, $u_{n_k}$ admits a
subsequence $u_{n_{k_j}}$ such that $u_{n_{k_j}}$ converges uniformly in
$A_\lambda$. 

\medskip

Therefore, the whole sequence $(u_n)$ converges uniformly to $0$ in 
$\Omega_\varepsilon$, and hence
$$ \int_{A_\varepsilon}  |x|^b |u_n|^p  \; dx < \frac{\varepsilon}{3}
\; \hbox{if} \;n \geq n_0(\varepsilon)$$
and hence $\| u_{n} \| < \varepsilon$ if $n \geq n_0$.

Therefore $u_n \to 0$ in $L^p_b(\R^n)$, and we conclude the proof.
\end{proof}

Using this theorem with $q=2$, the Lagrange multiplier rule and the symmetric 
criticality principle \cite{Palais} ($H^s_{2,a,rad}(\R^n)$ is the invariant 
subespace of $H^s_{2,a}(\R^n)$ under the action of the orthogonal group $O(n)$,
which is a compact Lie group), we immediately get:

\begin{corollary}
Let $2<p<2^*_b=\frac{2(n+b)}{n-2s}$, $a<n$ and 
$$ a(p-2-2ps) + 4bs  < 
 2s(p-2)(n-1).  $$
Let us consider the minimization problem
\begin{multline*}
m= \inf \left\{ [u]_{H^s}^2   + \int_{\R^n} |x|^a u^2 \; dx 
\; : \; u \in H^{s}_{2,a,rad}(\R^n), 
\int_{\R^n} |x|^b (u^+)^p =1 \right\} 
\end{multline*}
Then $m=m(a,b,p,s)>0$ and is achieved. Hence, problem
$$ (-\Delta)^{s} u + |x|^a u = |x|^b u^{p-1} \quad u \in
H^s_{2,a,rad}(\R^n)$$
admits a non-negative radial solution.
\end{corollary}

Moreover, using symmetric versions of the Mountain pass theorem 
as in \cite{Sintzoff}, applied to the functional

$$ \Phi(u)= \frac{1}{2} \; [u]_{H^s}^2 + \frac{1}{q} \;
\| u \|_{L^q_a}^q - \frac{1}{p} \; \| u \|_{L^p_a}^p \quad u \in H^{s}_{q,a,rad}(\R^n)
$$

we can also prove the following result:
\begin{corollary}
Let $n\geq 3$, $q>1$, $s>n/2$, $\max(2,q)<p< 2^*_{b}= \frac{2(n+b)}{n-2s}$ and 
$$ a \, (p-2-2ps) + b \, (2qs - q +2 ) < 
 2s \, (p-q) \, (n-1) $$ 
Then the problem 
$$  (-\Delta)^{s} u + |x|^a |u|^{q-2} u = |x|^b u^{p-1} \quad u \in
H^s_{2,a,rad}(\R^n)$$
admits a non negative radial solution, and infinitely many radial solutions
$u_k$ such that
$\Phi(u_k) \to +\infty$
\end{corollary} 

The proofs of these corollaries are exactly as in \cite{Sintzoff}. Hence,
they are omitted.

\section{Proof of the Symmetry Breaking Result}

\label{sb-section}

In this section, we prove theorem \ref{theorem-symmetry-breaking}, on
symmetry breaking for problem \eqref{inaball} in a large ball. For the proof, 
we need the following lemma  which allows us to carry the arguments involving 
cut-off functions in the setting of the fractional Laplacian. 

\begin{lemma}
Let $0<s<\min(1,\frac{n}{2})$ and $\eta \in C^\infty(\R)$ such that $\eta=1$ on $(-\infty,1/2)$, $0 \leq
\eta \leq 1$ and $\hbox{sup}(\eta) \subset (-\infty,1]$.  Set
$$ \eta_R(x) = \eta (\left |x|/ R \right) $$
Then, if $u \in H^s(\R^n)$, then
$$ \eta_R \cdot u \to u \quad \hbox{in} \; H^s(\R^n) \; \hbox{as} \: R \to
\infty, $$
Moreover if $u \in H^s_{q,a}(\R^n)$ then
$$ \eta_R \cdot u \to u \quad \hbox{in} \; H^s_{q,a}(\R^n) \; \hbox{as} \: R \to
\infty. $$
\label{lemma-cut-off}
\end{lemma}

\begin{proof}
A simple proof of the first assertion is given in \cite{Raman}, which is based on a bounded convergence 
argument. The second assertion follows immediately, likewise (recall remark
\ref{remark-characterizations}).  
\end{proof}

Now we are ready to give the proof of the symmetry breaking result.

\begin{proof}
Following the idea of \cite{Sintzoff}, we consider the minimization problems:

$$ M(R) = \inf \left\{ 
[u]_{H^s}^2 + \int_{B_R} |x|^a u^2 \; dx 
\; : \; u \in H^{s}_{2,a,0}(B_R), 
\int_{B_R} |x|^b (u^+)^p =1 \right\}  $$

$$ m(R) = \inf \left\{ 
[u]_{H^s}^2 + \int_{B_R} |x|^a u^2 \; dx 
\; : \; u \in H^{s}_{2,a,0,rad}(B_R), 
\int_{B_R} |x|^b (u^+)^p =1 \right\}  $$

and the corresponding minimization problems for $\R^n$

$$ M(\infty) = \inf \left\{ 
[u]_{H^s}^2  + \int_{\R^n} |x|^a u^2 \; dx 
\; : \; u \in H^{s}_{2,a}(\R^n), 
\int_{\R^n} |x|^b (u^+)^p =1 \right\}  $$

$$ m(\infty) = \inf \left\{ 
[u]_{H^s}^2 + \int_{\R^n} |x|^a u^2 \; dx 
\; : \; u \in H^{s}_{2,a,rad}(\R^n), 
\int_{\R^n} |x|^b (u^+)^p =1 \right\}  $$

\noindent where the Gagliardo $[\cdot]_{H^s}$ is allways taken in $\R^n$, i.e. is
given by \eqref{Gagliardo-seminorm} in both cases. 

The main idea of the proof is to compare the minimum energy levels for those
problems.

\medskip

First we consider the minimization problems $m(R)$ and $M(R)$, for the ball
$B_R$. We note that we have a continuous embedding
$$ H^s_{2,a,0}(B_R) \hookrightarrow H^s_0(B_R) $$
hence $m(R)$ and $M(R)$ are well defined. Moreover since $p<2^*$, 
by Rellich theorem, the embedding 
$H^s_{2,a,0}(B_R) \hookrightarrow L^p(B_R)$ is compact. It follows that $m(R)$ and $M(R)$ are 
achieved, and are positive for every $R>0$.  Moreover, since for every $u \in
H^s_{2,a,0}(B_R)$, $u^+$ is in the same space (for checking that, the characterizations 
in remark \ref{remark-characterizations} are useful) and we have that
$$ [u^+]_{H^s} \leq [u]_{H^s}, $$
we see that they have non-negative minimizer, which by the Lagrange multiplier rule (and the 
symmetric  criticality principle in \cite{Palais} in the case of $m(R)$ ) are weak solutions of 
in the sense of definition \ref{definition-weak-solution} of
\be  \left\{
\begin{gathered}
(-\Delta)^{s} u + |x|^a u = \lambda  |x|^b u^{p-1} \quad \hbox{in} \; B_R \\
u \geq 0 \;  \hbox{a.e. in}  \; B_R, \quad
u \equiv 0 \;  \hbox{in} \; \R^n-B_R
\end{gathered}
\right. \label{inaball-lambda} 
\ee
for some $\lambda$. Using $u$ itself as a test function, we see that $\lambda=M(R)>0$ (respectively $\lambda=m(R)>0$), 
so since $p \neq 2$, by replacing $u$ by some positive multiple, we can get solutions of 
\eqref{inaball-lambda} with $\lambda=1$. Moreover,  
by the strong minimum principle (theorem
\ref{minimum-principle}) those minimizers are strictly positive almost everywhere in $B_R$.

\medskip

Next we consider the minimization problems $m(\infty)$ and $M(\infty)$ for
the whole space $\R^n$. In a similar way as before, it follows from theorem 
\ref{compact-embedding} (with $q=2$), that $m(\infty)$ is achieved and is 
positive. We remark that the condition $b-a < \sigma(n,s,2,a)(p-2)$ in that theorem is equivalent to
\eqref{compactness-condition}. 

\medskip

However, we claim that $M(\infty)=0$. To see that, choose $u \in
C_0^\infty(B_1)$, $u \not \equiv 0$, $u \geq 0$ and consider the translated functions
$$ u_t(x)= u(x-t v) $$
for a fixed unitary vector $v\in \R^n$. Then,
\be M(\infty) \leq 
\frac{[u_t]_{H_s}^2 + \int_{\R^n} |x|^a u_t^2 \; dx }
{\left( \int_{\R^n} |x|^b u_t^p \; dx \right)^{2/p}  } 
\label{M-infty-bound}. \ee

We have the following estimates
$$ [u_t]_{H^s} = [u]_{H^s}, $$

$$ \int_{\R^n}  |x|^a u_t^2 \; dx =
\int_{B_1} |y+t v|^a u^2(y) \; dy
\leq (1+t)^a \int_{B_1} u^2(y) \; dy, $$

$$  \int_{\R^n} |x|^b u_t^p \; dx = 
\int_{B_1} |y+t v|^b u^p(y) \; dy
\geq (t-1)^b \int_{B_1} u^p(y) \; dy. $$

Hence, since $ 2b > ap $, we see that $M(\infty)=0$ by letting $t \to
+\infty$ in \eqref{M-infty-bound}.

\medskip

We claim that
\be \lim_{R \to +\infty} M(R)= M(\infty) \; \label{limit-M}, \ee

\be \lim_{R \to +\infty} m(R)= m(\infty) \; \label{limit-m}. \ee

\noindent It is clear that $M(R) \geq M(\infty)$, $m(R) \geq m(\infty)$. On the other
hand, let us prove \eqref{limit-M}: given $\varepsilon>0$, choose $u \in H^{s}_{2,a}(\R^n)$ such that
$$  \int_{\R^n} |x|^b (u^+)^p =1 $$
and
$$ E(u) := [u]_{H^s}^2 + \int_{\R^n} |x|^a u^2 \; dx < M(R) + \varepsilon. $$
Consider the cut-off functions $\eta_R$ constructed 
in lemma \ref{lemma-cut-off}. Then, from that lemma, 
$$ E(\eta_R \cdot u) \to E(u) \; \hbox{as} \; R \to +\infty. $$
Since $\eta_R \cdot u \in H^s_0(B_R)$, $ M(R)  \leq M(\eta_R u) $.
We conclude that
$$ M(R) \leq M(\infty) + \varepsilon \; \hbox{for} \; R \geq
R_0(\varepsilon).$$
This proves \eqref{limit-M}. The proof of  \eqref{limit-m} is similar. 

\medskip

It follows that if $R$ is large enough, $M(R) < m(R)$, and hence for $R$
large enough, the positive minimizers that we have found for $m(R)$ and
$M(R)$ are different. Hence, problem \eqref{inaball} admits both a radial positive solution, 
and a non-radial one.

\end{proof}


{\bf Acknowledgments}
I want to thank L. Del Pezzo, N. Wolanski and A. Salort for several interesting discussions on the fractional Laplacian, 
and E. J. Lami Dozo for bringing the symmetry breaking phenomenon and the reference \cite{Sintzoff} to my attention.

\end{document}